\documentclass[oneside, 12pt]{article}
\usepackage[margin=3cm]{geometry}

\pdfoutput=1 

\usepackage[utf8]{inputenc}
\usepackage[english]{babel}

\usepackage{graphicx}
\usepackage{amsthm}   
\usepackage{amsmath}  
\usepackage{amssymb}  
\usepackage{mathabx} 
\usepackage{bm} 
\usepackage{verbatim} 
\usepackage{enumitem} 
\usepackage{cancel}   
\usepackage{multicol}
\usepackage[dvipsnames]{xcolor} 
\usepackage{xspace}
\usepackage{xurl}
\usepackage{orcidlink} 

\usepackage{hyperref}
\hypersetup{
  colorlinks=true,
  citecolor=black,
  linkcolor=black,
  urlcolor=black,
  breaklinks=true,
  bookmarksnumbered=true,
  bookmarksopen=true,
  pdftitle={Strictly Positive Fragments of the Provability Logic of Heyting Arithmetic},
  pdfauthor={Ana de Almeida Borges and Joost J. Joosten},
  pdfsubject={},
  pdfcreator={},
  pdfkeywords={}
}
\usepackage[numbers]{natbib}   

\theoremstyle{definition}
\newtheorem{theorem}{Theorem}[section]
\newtheorem{definition}[theorem]{Definition}
\newtheorem{lemma}[theorem]{Lemma}
\newtheorem{corollary}[theorem]{Corollary}

\newtheorem{remark}[theorem]{Remark}

\numberwithin{equation}{section}

\newcommand{\PL}{\mathsf{PL}}
\newcommand{\PLsp}{\mathsf{PL}^+}
\newcommand{\TPL}{\mathsf{TPL}}
\newcommand{\QPL}{\mathsf{QPL}}
\newcommand{\QPLsp}{\mathsf{QPL}^+}
\newcommand{\TQPL}{\mathsf{TQPL}}

\newcommand{\Sh}{\mathcal{S}} 
\newcommand{\Lang}{\mathcal{L}} 
\newcommand{\Lb}{\Lang_\Box} 
\newcommand{\Lbf}{\Lang_{\Box, \forall}} 
\newcommand{\Lp}{\Lang^+_{\Box, \forall}} 
\newcommand{\Lbp}{\Lang^+_\Box} 
\newcommand{\Lpoly}[1]{\Lang_{[\alpha], \alpha < #1}}
\newcommand{\Lpolyp}[1]{\Lang^+_{[\alpha], \alpha < #1}}

\newcommand{\pfs}{^{\star}} 
\newcommand{\pfsT}{^{\star_{T}}} 
\newcommand{\pfsPA}{^{\star_{\PA}}} 
\newcommand{\pfsHA}{^{\star_{\HA}}} 

\newcommand{\fs}{^{*}} 
\newcommand{\fsT}{^{*_{T}}} 
\newcommand{\fsHA}{^{*_{\HA}}} 

\newcommand{\ofs}{^{\circledast}} 
\newcommand{\ofsT}{^{\circledast_{T}}} 
\newcommand{\ofsPA}{^{\circledast_{\PA}}} 
\newcommand{\ofsHA}{^{\circledast_{\HA}}} 

\newcommand{\is}{^{\circ}} 
\newcommand{\isT}{^{\circ_{\tau}}} 
\newcommand{\isHA}{^{\circ_{\eta}}} 

\newcommand{\ois}{^{\circledcirc}} 
\newcommand{\oisT}{^{\circledcirc_{\tau}}} 
\newcommand{\oisHA}{^{\circledcirc_{\eta}}} 

\newcommand{\four}{\normalfont{\textbf{4}}}

\newcommand{\fE}{\bm{\forall}\normalfont{\textbf{E}}}
\newcommand{\fI}{\bm{\forall}\normalfont{\textbf{I}}}

\newcommand{\Prop}{\normalfont{\textsc{Prop}}} 
\newcommand{\Rel}{\normalfont{\textsc{Rel}}} 

\newcommand{\CPC}{\mathsf{CPC}}
\newcommand{\IPC}{\mathsf{IPC}}
\newcommand{\K}{\mathsf{K}}
\newcommand{\Kf}{\mathsf{K4}}
\newcommand{\Kfthree}{\mathsf{K4.3}}

\newcommand{\T}{\mathsf{T}}
\newcommand{\Sfour}{\mathsf{S4}}
\newcommand{\Sfourthree}{\mathsf{S4.3}}
\newcommand{\Sfive}{\mathsf{S5}}
\newcommand{\GL}{\mathsf{GL}}
\newcommand{\iGL}{\mathsf{iGL}}
\newcommand{\GLthree}{\mathsf{GL.3}}
\newcommand{\GLS}{\mathsf{GLS}}
\newcommand{\GLP}{\mathsf{GLP}}
\newcommand{\GLPzero}{\mathsf{GLP}^0}

\newcommand{\RC}{\mathsf{RC}}
\newcommand{\RCzero}{\mathsf{RC}^0}

\newcommand{\QRC}{\mathsf{QRC}_1}
\newcommand{\QGL}{\mathsf{QGL}}
\newcommand{\QK}{\mathsf{QK}}
\newcommand{\QKf}{\mathsf{QK}4}

\newcommand{\Q}{\mathsf{Q}}

\newcommand{\IPL}{\mathsf{RL}}
\newcommand{\IQPL}{\mathsf{QRL}}

\newcommand{\PSPACE}{\mathsf{PSPACE}}
\newcommand{\PTIME}{\mathsf{PTIME}}
\newcommand{\coNP}{\mathsf{coNP}}

\newcommand{\DC}{\mathcal{DC}} 
\newcommand{\DCone}{\DC(\Pi_1, \Sigma_1)} 

\makeatletter
\newcommand{\inspect}[1]{%
  \def\inspectspace{\mskip3mu\relax}
  \sbox\z@{$#1$}
  \sbox\tw@{\thickmuskip=0mu$#1$}
  \ifdim\wd\tw@<\wd\z@ \def\inspectspace{\mskip9mu\relax}\fi 
}
\makeatother

\DeclareMathOperator{\Existsa}{\exists}
\newcommand{\Exists}[1]{
  \inspect{#1}
  \Existsa #1 \inspectspace
}

\DeclareMathOperator{\Foralla}{\forall}
\newcommand{\Forall}[1]{
  \inspect{#1}
  \Foralla #1 \inspectspace
}

\newcommand{\F}{\mathcal{F}}
\newcommand{\M}{\mathcal{M}}

\newcommand{\la}{\langle}
\newcommand{\ra}{\rangle}

\newcommand{\N}{\mathbb{N}}
\newcommand{\PA}{\mathsf{PA}}
\newcommand{\HA}{\mathsf{HA}}
\newcommand{\EA}{\mathsf{EA}}
\newcommand{\iEA}{\mathsf{iEA}}

\newcommand{\fv}{\normalfont{\text{fv}}} 
\newcommand{\mdepth}{\normalfont{\text{d}}_\Diamond} 

\renewcommand{\vec}[1]{\bm{#1}}

\newcommand{\gnum}[1]{\ulcorner #1 \urcorner} 

\newcommand{\xaltern}[1]{\sim_{#1}} 

\newcommand{\subst}[2]{[#1 {\leftarrow} #2]} 

\DeclareMathOperator{\umod}{mod} 
\newcommand{\bij}[1]{\lcorners #1 \rcorners} 
\title{Strictly Positive Fragments of\\ the Provability Logic of Heyting Arithmetic}
  \author{
    Ana de Almeida Borges%
    \thanks{\url{ana.agvb@gmail.com}}
    \ \orcidlink{0000-0001-5152-198X}
  \and
    Joost J. Joosten%
    \thanks{\url{jjoosten@ub.edu}}
    \ \orcidlink{0000-0001-9590-5045}
  }
  \date{University of Barcelona}

\begin{document}

\maketitle

\begin{abstract}
  We determine the strictly positive fragment $\QPLsp(\HA)$ of the quantified provability logic $\QPL(\HA)$ of Heyting Arithmetic. We show that $\QPLsp(\HA)$ is decidable and that it coincides with $\QPLsp(\PA)$, which is the strictly positive fragment of the quantified provability logic of of Peano Arithmetic. This positively resolves a previous conjecture of the authors described in \cite{Escape}.

On our way to proving these results, we carve out the strictly positive fragment $\PLsp(\HA)$ of the provability logic $\PL(\HA)$ of Heyting Arithmetic, provide a simple axiomatization, and prove it to be sound and complete for two types of arithmetical interpretations. 

The simple fragments presented in this paper should be contrasted with a recent result by Mojtahedi \cite{Mojtahedi2022}, where an axiomatization for $\PL(\HA)$ is provided. This axiomatization, although decidable, is of considerable complexity.
\end{abstract}

\section{Introduction}

The provability logic of Heyting Arithmetic ($\HA$) has recently been axiomatized by Mojtahedi \cite{Mojtahedi2022}, solving a fifty-year-old problem. This axiomatization is decidable but also quite technically complex.

One way of simplifying a logic $L$ without losing all of its expressive power is to restrict its language to the strictly positive fragment $L^+$. This is the fragment containing only $\top$, propositional or predicate symbols (depending on whether $L$ is propositional or quantified, respectively), conjunctions, universal quantifiers (if $L$ is quantified), and diamonds (if $L$ is modal). Theorems of $L^+$ are implications between strictly positive formulas. For example, $\Forall x \Diamond P(x) \vdash_{L^+} \Diamond \Forall x P(x)$ is a grammatical statement that might or not hold depending on $L^+$.
There are several examples of pairs $\la L, L^+ \ra$ where $L^+$ has lower complexity than $L$ (see Table~\ref{table:sp} in Section~\ref{sec:sp} for a survey), and for at least some of them $L^+$ has been actively studied and used in its own right.

In \cite{Escape}, we studied the strictly positive fragment of the quantified provability logic of Peano Arithmetic ($\PA$) and conjectured that it might coincide with the strictly positive fragment of the quantified provability logic of $\HA$.
In the present paper, we show that this is indeed the case, by providing a link between a $\PA$-extension and an $\HA$-extension of a given arithmetical realization.

We further study the strictly positive fragment of the (propositional) provability logic of $\HA$, and show that it is quite simple and also coincides with the strictly positive fragment of the provability logic of $\PA$.

In what follows, we introduce the main concepts related to provability logics, providing some more context and history, and end this section with a more detailed plan of the paper.

\subsection{Provability logics}

The arithmetical predicate $\Box_T(\cdot)$ has been often used ever since Gödel. This predicate is a natural and $\Sigma_1$-complete formalization of the concept of provability in the theory $T$, in that $\N \vDash \Box_T(n)$ holds exactly when $n$ is the Gödel number of a formula that is provable in $T$ (as long as $T$ is a reasonable enough theory, for example as long as $T$ has a decidable set of axioms).

The provability logic $\PL(T)$ of a theory $T$ describes the structural behavior of $\Box_T(\cdot)$ as can be proven in $T$ itself. This structural behavior is expressed in a propositional modal logic where, for example, we have $(\Box(p \land q) \to \Box p) \in \PL(T)$, since, for every possible arithmetical reading $p\pfs$ and $q\pfs$ of $p$ and $q$, respectively, we have that $T\vdash \Box_T(p\pfs \land q\pfs) \to \Box_T p\pfs$. In a similar spirit, one defines the set $\TPL(T)$ to be the set of those propositional logical formulas that induce corresponding formulas that are always true on the standard model of arithmetic.

The definition of $\PL(T)$ is of complexity $\Pi_2$ (``for all readings $\cdot\pfs$, it is provable that $\varphi\pfs$''), and the definition of $\TPL(T)$ is not even arithmetical (``for all readings $\cdot\pfs$, it is true that $\varphi\pfs$).
However, it turns out that both $\PL(T)$ and $\TPL(T)$ can be generated by the decidable modal logics $\GL$ (thus named in reference to Gödel and Löb) and $\GLS$ (thus named in reference to $\GL$ and Solovay) respectively, for a large class of classical logical theories $T$, including Peano Arithmetic. See \cite{Boolos1993} for a more thorough overview.

\subsection{Quantified provability logic and its fragments}

The quantified provability logic $\QPL(T)$ of a theory $T$ is defined in an analogous way to $\PL(T)$.
We restrict ourselves to the languages of predicate logic without function symbols or identity.

As an example, the converse Barcan formula is valid:
\begin{equation*}
  (\Box_T \Forall x \varphi(x) \to \Forall x \Box_T \varphi(\dot x))
  \in
  \QPL(T)
  ,
\end{equation*}
but the Barcan formula is not:
\begin{equation*}
  (\Forall x \Box_T \varphi(\dot x) \to \Box_T \Forall x \varphi(x))
  \notin
  \QPL(T).
\end{equation*}

Recall that, in the classical case, $\PL(T)$ allows for a reduction from a $\Pi_2$-definition to a decidable one. Vardanyan proved in \cite{Vardanyan1986} that this is not the case for $\QPL(\PA)$, which is actually $\Pi_2$-complete. George Boolos and Vann McGee \cite{BoolosMcGee1987} proved that $\TQPL(\PA)$ is $\Pi_1$-complete in the set of true arithmetic. Vardanyan's result allows for various generalizations \cite{Vardanyan1988, Berarducci1989, VisserAndDeJonge:2006:NoEscape} and for a few exceptions of very limited range \cite{ArtemovJaparidze1990, Yavorsky2002, HaoTourlakis2021}.

In \cite{Escape}, the strictly positive fragment $\QPLsp(\PA)$ of the quantified provability logic of $\PA$ was presented and proven to be decidable, axiomatized by a logic called $\QRC$. The same paper proved soundness of $\QRC$ for $\HA$ and conjectured completeness. The current paper settles that conjecture positively and proves completeness of $\QRC$ for $\HA$ for two different arithmetical readings: the \emph{finitary reading} and the \emph{infinitary reading}. 

The finitary reading translates predicates to arithmetical formulas with the same arity, as is usual in the case of quantified provability logic. The infinitary reading falls in the tradition so-called reflection calculi. Propositional reflection calculus translates propositional variables to elementary axiomatizations of theories. Generalizing this to predicate logic, the infinitary reading translates a predicate $S$ to the axiomatization of a collection of theories parametrized by a number of parameters equal to the arity of~$S$.

\subsection{Constructive provability logics}

For about fifty years, the provability logic of Heyting Arithmetic remained an open problem. Recently, Mojtahedi presented a solution, with a complete axiomatization of $\PL(\HA)$ described in \cite{Mojtahedi2022}.
Though this axiomatization is proven to be decidable, it is of considerable complexity.
To give an impression of the nature of $\PL(\HA)$, we mention an incomplete selection of landmark results here and refer the reader to \cite{arte:prov04} for a more detailed overview.

De Jongh proved in \cite{dejo:maxi70} the non-trivial result that the propositional fragment of $\PL(\HA)$ is exactly $\IPC$, the set of intuitionistically valid propositional formulas.\footnote{The result is not as stable as in the classical logical case and, for example, the propositional fragment of the provability logic of $\HA$ together with Markov's principle and with Church's thesis is unknown.}
It was easy to observe that $\iGL \subseteq \PL(\HA)$, where $\iGL$ is exactly as $\GL$ except that it is defined over $\IPC$ instead of over the classical propositional logic $\CPC$.

Since $\HA$ has different meta-mathematical properties than $\PA$, different provability logical principles were expected to be found. For example, Markov's rule is admissible for $\HA$ in that $\HA \vdash \neg \neg \pi$ implies $\HA \vdash \pi$ whenever $\pi \in \Pi_2$.
Then the admissibility of Markov's rule translates to $(\Box \neg \neg \Box \varphi \to \Box \Box \varphi) \in \PL(\HA)$,
since $\Box_{\HA} \varphi$ is always $\Sigma_1$ regardless of $\varphi$.
More generally:
\begin{equation*}
  \left(
    \Box \neg \neg \left(\Box \varphi \to \bigvee_i \Box \varphi_i\right)
    \to
    \Box \left(\Box \varphi \to \bigvee_i \Box \varphi_i\right)
  \right) \in \PL(\HA)
  .
\end{equation*}

There is actually a more general relationship between $\PL(\HA)$ and admissible rules. Visser showed that the admissible rules of $\HA$ and the admissible rules of $\IPC$ are the same so that each admissible rule $\frac{\varphi}{\psi}$ of $\HA$ gives rise to $(\Box \varphi \to \Box \psi) \in \PL(\HA)$.
Although Rybakov showed in \cite{ryba:admi97} that the collection of admissible rules of $\IPC$ is decidable, it was only later that concrete characterizations were found. In particular, Iemhoff showed in \cite{iemh:admi01} how the admissible rules can be defined in terms of Visser's rule.

The exact definition of Visser's rule is not of real importance in the context of this paper.
However, we wish to convey an impression of the complexity of $\PL(\HA)$ and as such we find it instructive to formulate it.
For an arbitrary natural number $n$, using the abbreviation
$
  (\varphi)(\psi_1, \hdots, \psi_n) := (\varphi \to \psi_1) \lor \cdots \lor (\varphi \to \psi_n)
$
and letting $\alpha := \bigwedge_{i=1}^n (\epsilon_i \to \zeta_i)$, Visser's rule $({\sf V_n})$ is defined as:
\begin{equation*}
  \frac
    {\left(\alpha \to (\beta \lor \gamma) \right) \lor \delta}
    {(\alpha)(\epsilon_1, \hdots, \epsilon_n, \beta, \gamma) \lor \delta}
    ,
\end{equation*}
leading to:
\begin{equation*}
  \left(
    \Box \left( \left(\alpha \to (\beta \lor \gamma)\right) \lor \delta\right)
    \to
    \Box \left( (\alpha)(\epsilon_1, \hdots, \epsilon_n, \beta, \gamma) \lor \delta \right)
  \right) \in \PL(\HA)
  .
\end{equation*}

So far we have seen how various meta-mathematical properties led to new principles of the form $(\Box \varphi \to \Box \psi)$ in $\PL(\HA)$. However, not every meta-mathematical property of $\HA$ directly translates into a principle in $\PL(\HA)$. For example, Myhill~\cite{myhill73} and Friedman~\cite{frie:disj75} observed that the disjunction property of $\HA$ is not provable in $\HA$ itself, whence $(\Box (p \lor q) \to \Box p \lor \Box q) \notin \PL(\HA)$. Notwithstanding this negative result, Leivant
could prove in \cite{leiv:abso75,viss:subs02} a weak version of the disjunction property to the extent that
$(\Box (p\lor q) \to \Box ( p \lor \Box q)) \in \PL(\HA)$, whence also $(\Box (p\lor q) \to \Box ( \Box p \lor \Box q)) \in \PL(\HA)$.

Visser characterized the closed fragment of $\PL(\HA)$ in \cite{viss:eval85} and the ``logic'' of $\Sigma_1$-substitutions of $\HA$ in \cite{viss:subs02}. The latter turned out to be important for Mojtehedi's result, in a sense, reflecting that the substitutions in Solovay's completeness proof are Boolean combinations of $\Sigma_1$-formulas.
In \cite{Mojtahedi2022}, Mojtaba Mojtahedi could thus include and extend all the previous results, culminating in a complete and decidable axiomatization of $\PL(\HA)$. This axiomatization is $\iGL$ together with axioms of the form $\Box \varphi \to \Box \psi$, where $\varphi$ and $\psi$ satisfy a complicated yet decidable condition.

The first result of this paper is a simple $\PTIME$-decidable fragment of $\PL(\HA)$: the strictly positive fragment $\PL(\HA)^+$. The second result of this paper is that the first result can be extended to the quantified setting, leading to a decidable (although to our knowledge not in $\PTIME$) fragment of $\QPL(\HA)$. Since the propositional logic $\PL(\HA)$ was already such a challenge, there has been hardly any work done on the quantified provability logic $\QPL(\HA)$ of $\HA$ before now.

\subsection{Plan of the paper}

In Section \ref{section:ModalLogicsAndStrict} we introduce the most relevant modal logics for this paper, both propositional and quantified.
For each modal logic, we define its corresponding strictly positive fragment. As the main contribution of this section, we display some literature findings showing how in general a shift from the full logic to the strictly positive fragment comes with a drop in computational complexity of the problem of theoremhood.

In Section \ref{section:SemanticsModalLogics} we revisit relational and arithmetical semantics for modal provability logics for both the propositional and the predicate logical case. We state the main theorems relating the logics and their semantics. 
All of the results in this section are known, but they are needed for the sequel. The only novel detail is the exposition of $\QRC$ in terms of Kripke sheaves as in \cite{BorgesPhD}.

In Section \ref{section:hierarchyInHA} we prepare for a move to the arithmetical semantics of provability logics for intuitionistic theories by first analyzing the behavior of fragments of the arithmetical hierarchy in the constructive setting.

In Section \ref{section:Solovay} we revisit the main ingredients of the arithmetical completeness proof for provability logic and observe that the needed substitutions are of low and constructively well-behaved complexity.

Sections \ref{section:PLHA+} and \ref{section:QPLplusHA} contain the main results of the paper, providing an axiomatization of the strictly positive fragment of the provability logic of Heyting Arithmetic, both in the propositional and in the quantified case.

Finally, Section \ref{section:InterpolationAndFP} closes off with some observations on the interpolation and the fixpoint theorem for these fragments.

\section{Modal logics and strictly positive fragments}\label{section:ModalLogicsAndStrict}

In this section we introduce the most relevant modal logics for this paper. These are mostly related to the provability logic $\GL$ (see \cite{Boolos1993}). We dwell on both propositional logics and on their corresponding predicate logical versions. For each logic, we define its strictly positive fragment and display some findings from the literature showing how in general a shift from the full logic to its strictly positive fragment comes with a drop in the computational complexity of theoremhood.

\subsection{Propositional modal logics}

The language of propositional modal logic $\Lb$ is obtained from the language of propositional logic by extending it with a unary $\Box$ symbol:
  \begin{equation*}
    \varphi ::= \bot \mid p \mid (\varphi \to \varphi) \mid (\varphi \land \varphi)  \mid (\varphi \lor \varphi)  \mid \Box \varphi
    ,
  \end{equation*}
where $p$ represents a propositional symbol.
  In classical logic, $\bot$ and $\to$ suffice to define $\land$ and $\lor$, but not so in constructive logic. Negation is defined as usual, with $\neg \varphi := \varphi \to \bot$, and the diamond operator is defined as the dual of the box: $\Diamond \varphi := \neg \Box \neg \varphi$.

The only (non-strictly positive) modal logics that we consider in this paper are all normal modal logics. Thus, they are an extension of the basic normal modal logic $\K$.
The modal logic $\K$ has as axioms all propositional logical tautologies in $\Lb$ and all distribution axioms of the form $\Box (\varphi \to \psi) \to (\Box \varphi \to \Box \psi)$. The rules are \emph{modus ponens} and necessitation, to wit, $\K \vdash \varphi \implies \K \vdash \Box \varphi$.

The modal logic $\Kf$ is defined from $\K$ by adding all the transitivity axioms (also called $\four$), which are of the form $\Box \varphi \to \Box \Box \varphi$. The modal logic $\GL$ is defined by adding Löb's axiom scheme $\Box (\Box \varphi \to \varphi) \to \Box \varphi$ to $\Kf$ (or, equivalently, to $\K$).

For a given ordinal $\Lambda$, the modal logic $\GLP_\Lambda$ \cite{Japaridze1988, Beklemishev2005} with language $\Lpoly\Lambda$ has a modality $[\alpha]$ for each $\alpha < \Lambda$ (with corresponding dual $\la \alpha \ra$) and is defined such that each modality $[\alpha]$ behaves like the $\Box$ of $\GL$ and furthermore the following two axiom schemes are available, both with $\beta < \alpha$: monotonicity $[\beta] \varphi \to [\alpha] \varphi$; and negative introspection $\la \beta \ra \varphi \to [\alpha]\la \beta \ra \varphi$.

\subsection{Quantified modal logics}

The language of modal logic can be extended to a quantified (first-order) setting in the natural way. We describe a simple version with no equality nor function symbols, $\Lbf$. Fix a signature $\Sigma$, which in this case is simply a set of relation symbols $\Rel_\Sigma$ and their respective arities (we mostly omit specific mentions of $\Sigma$). Then $\Lbf$ is defined as follows:
\begin{equation*}
  \varphi ::= \bot \mid S(\vec x) \mid (\varphi \to \varphi) \mid (\varphi \land \varphi) \mid (\varphi \lor \varphi) \mid \Forall{x} \varphi \mid \Exists x \varphi \mid \Box \varphi,
\end{equation*}
where $S \in \Rel_\Sigma$ is a relation symbol with arity $n$, $\vec x$ is an $n$-tuple of variables, and $x$ is a variable. Again, $\neg \varphi := \varphi \to \bot$ and $\Diamond \varphi := \neg \Box \neg \varphi$. In classical settings, conjunction, disjunction and existential quantification can be defined from the other connectives.

The free variables of a formula $\varphi$ are defined as usual, and denoted by $\fv(\varphi)$. The expression $\varphi\subst{x}{y}$ denotes the formula $\varphi$ with all free occurrences of the variable $x$ simultaneously replaced by the variable $y$. We say that $y$ is free for $x$ in $\varphi$ if no occurrence of $y$ becomes bound in $\varphi\subst{x}{y}$.

Given a propositional modal logic $L$, one can define its quantified version $\Q L$ as described in \cite{HughesCresswell1996} by adding the following axiom scheme and rule to the first order formulation of $L$: $\Forall{x} \varphi \to \varphi\subst{x}{y}$ with $y$ free for $x$ in $\varphi$ ($\fE$); and $\Q L \vdash \varphi \to \psi \implies \Q L \vdash \varphi \to \Forall x \psi$ with $x \notin \fv(\varphi)$ ($\fI$).
This leads to the straightforward definition of logics such as $\QK$, $\QKf$, and $\QGL$.

As is to be expected, quantified modal logics such as $\QK$ and $\QKf$ are undecidable (they typically contain plain predicate logic). Even small fragments of quantified modal logics are undecidable, such as the fragment with only two variables \cite{Kontchakov2005} and the one with only one monadic predicate symbol and two variables \cite{RybakovShkatov2019}. There are, however, some decidable fragments of these logics discussed in \cite{WolterZakharyaschev2001}. Another decidable fragment is given by the strictly positive quantified modal logic $\QRC$ \cite{QRC1}, further described below.

\subsection{Strictly positive logics}
\label{sec:sp}

Propositional strictly positive formulas (abbreviated p.s.p.~formulas) are built up from $\top$, propositional variables, the binary $\land$ and the unary $\Diamond$ (or $\la \alpha \ra$ in the polymodal setting). Thus, the strictly positive propositional language $\Lbp$ is defined as $\varphi ::= \top \mid p \mid (\varphi \land \varphi) \mid \Diamond \varphi$, and the polymodal version $\Lpolyp\Lambda$ is defined as $\varphi ::= \top \mid p \mid (\varphi \land \varphi) \mid \la \alpha \ra \varphi$, with $\alpha < \Lambda$.
We extend the definition to strictly positive quantified modal formulas over some given relational vocabulary $\Sigma$ as $\varphi ::= \top \mid S(\vec t) \mid (\varphi \land \varphi) \mid \Forall{x} \varphi \mid \Diamond \varphi$, where $S$ is in $\Rel_\Sigma$, and denote this language by $\Lp$. For practical reasons related with the study of $\QRC$, the signature $\Sigma$ may allow for individual constants in the strictly positive case, and so terms $\vec t$ are either variables or constants. We do not work with function symbols.

Note that the above are the only available connectives. The absence of implication and negation does not allow us to define any of the other ones.
The statements of a strictly positive logic are implications between two strictly positive formulas, typically written $\varphi \vdash \psi$ (where $\varphi, \psi$ are strictly positive) instead of $\varphi \to \psi$ to remind ourselves that the implication cannot be nested.

An important propositional strictly positive modal logic in the context of this paper is the Reflection Calculus $\RC_\Lambda$, defined as follows.

\begin{definition}[$\RC_\Lambda$ \cite{Dashkov2012, Beklemishev2012, FernandezDuqueJJJ2014}]
  
  For $\varphi, \psi$ and $\chi$ formulas in the language $\Lpolyp\Lambda$ and $\alpha, \beta < \Lambda$, the axioms and rules of the Reflection Calculus (denoted by $\RC_\Lambda$) are the following:
\begin{enumerate}[label=\upshape(\roman*),ref=\thetheorem.(\roman*)]
    \item $\varphi \vdash_{\RC_\Lambda} \varphi$ and $\varphi \vdash_{\RC_\Lambda} \top$;
    \item $\varphi \land \psi \vdash_{\RC_\Lambda} \varphi$ and $\varphi \land \psi \vdash_{\RC_\Lambda} \psi$ (conjunction elimination);
    \item if $\varphi \vdash_{\RC_\Lambda} \psi$ and $\varphi \vdash_{\RC_\Lambda} \chi$, then $\varphi \vdash_{\RC_\Lambda} \psi \land \chi$ (conjunction introduction);
    \item if $\varphi \vdash_{\RC_\Lambda} \psi$ and $\psi \vdash_{\RC_\Lambda} \chi$, then $\varphi \vdash_{\RC_\Lambda} \chi$ (cut);
    \item if $\varphi \vdash_{\RC_\Lambda} \psi$, then $\la \alpha \ra \varphi \vdash_{\RC_\Lambda} \la \alpha \ra \psi$ (necessitation);
    \item $\la \alpha \ra \la \alpha \ra \varphi \vdash_{\RC_\Lambda}  \la \alpha \ra \varphi$ (transitivity);
    \item $\la \alpha \ra \varphi \vdash_{\RC_\Lambda}  \la \beta \ra \varphi$, for $\alpha > \beta$ (monotonicity);
    \item $\la \alpha \ra \varphi \land \la \beta \ra \psi \vdash_{\RC_\Lambda}  \la \alpha \ra (\varphi \land \la \beta \ra \psi)$ for $\alpha > \beta$ (negative introspection).
  \end{enumerate}
\end{definition}

The logic $\RC_\Lambda$ is also called the strictly positive fragment of $\GLP_\Lambda$ since one can prove that for any p.s.p.~formulas $\varphi$ and $\psi$:
\begin{equation*}
  \GLP_\Lambda \vdash \varphi \to \psi \iff \varphi \vdash_{\RC_{\Lambda}} \psi
  .
\end{equation*}
See \cite{Dashkov2012, Beklemishev2012, FernandezDuqueJJJ2014, BeklemishevFernandezDuqueJJJ2014} for details.

The logic $\RC_1$ has just one modality and is proven in \cite{Dashkov2012} to be the strictly positive fragment of any modal logic between $\Kf$ and $\GL$, inclusive.

In general, when considering a strictly positive logic $P$, one may ask whether there is some modal logic $L$ whose strictly positive fragment $L^+$ coincides with $P$. In that case we would have:
\begin{equation}
\label{eq:Lsp}
  \varphi \vdash_P \psi
  \iff
  L \vdash \varphi \to \psi
  ,
\end{equation}
where $\varphi$ and $\psi$ are strictly positive formulas. In the remainder of this paper, given a logic $L$ we shall denote its strictly positive fragment by $L^+$.

\begin{definition}
For a modal logic $L$ the \emph{strictly positive fragment} $L^+$ of $L$ is defined as:
\begin{equation*}
  L^+ := \{ \la\varphi, \psi\ra \mid L \vdash \varphi \to \psi \text{ and $\varphi$ and $\psi$ are strictly positive formulas} \}.
\end{equation*}
\end{definition}

One can also consider strictly positive fragments for quantified modal logics. Of principal interest in this paper is the Quantified Reflection Calculus with one modality $\QRC$, which is an extension of $\RC_1$.

\begin{definition}[$\QRC$ \cite{QRC1}]
  Let $\varphi$, $\psi$, and $\chi$ be strictly positive formulas in $\Lp$ with signature $\Sigma$. The axioms and rules of $\QRC$ are the following:

\begin{enumerate}[label=\upshape(\roman*),ref=\thetheorem.(\roman*)]
  \item $\varphi \vdash \top$ and $\varphi \vdash \varphi$;
  \item $\varphi \land \psi \vdash \varphi$ and $\varphi \land \psi \vdash \psi$ (conjunction elimination);
  \item if $\varphi \vdash \psi$ and $\varphi \vdash \chi$, then $\varphi \vdash \psi \land \chi$ (conjunction introduction);
  \item if $\varphi \vdash \psi$ and $\psi \vdash \chi$, then $\varphi \vdash \chi$ (cut);\label{rule:cut}
  \item if $\varphi \vdash \psi$, then $\Diamond \varphi \vdash \Diamond \psi$ (necessitation);\label{rule:necessitation}
  \item $\Diamond \Diamond \varphi \vdash \Diamond \varphi$ (transitivity); \label{ax:transitivity}
  \item if $\varphi \vdash \psi$, then $\varphi \vdash \Forall{x} \psi$, for $x \notin \fv(\varphi)$ (quantifier introduction on the right); \label{rule:forallR}
  \item if $\varphi\subst{x}{t} \vdash \psi$, then $\Forall{x} \varphi \vdash \psi$, for $t$ free for $x$ in $\varphi$ (quantifier introduction on the left);\label{rule:forallL}
  \item if $\varphi \vdash \psi$, then $\varphi\subst{x}{t} \vdash \psi\subst{x}{t}$, for $t$ free for $x$ in $\varphi$ and $\psi$ (term instantiation);\label{rule:term_instantiation}
  \item if $\varphi\subst{x}{c} \vdash \psi\subst{x}{c}$, then $\varphi \vdash \psi$, for $c$ not in $\varphi$ nor $\psi$ (constant elimination).\label{rule:constants}
\end{enumerate}
\end{definition}

Analogously to the propositional case, $\QRC$ is the strictly positive fragment of any logic between $\QKf$ and $\QGL$, inclusive, and also the strictly positive fragment of the quantified provability logic of Peano Arithmetic \cite{Escape}.
In this paper, we show that $\QRC$ is the strictly positive fragment of the quantified provability logic of Heyting Arithmetic as well.

Moving from a logic to its strictly positive fragment generally leads to a reduction in the computational complexity of the corresponding decision procedures. Many such reductions are listed in Table~\ref{table:sp}, although the strictness of most reductions depends on complexity theoretical assumptions such as $\PSPACE \neq \PTIME$. Notably, such assumptions are not required in the case of $\QRC$.

\begin{table}[ht]
\centering
\caption{Complexity of some modal logics and their strictly positive counterparts.}
\label{table:sp}
\begin{tabular}{lll|lll}
  \hline
  \multicolumn{3}{c|}{Full logic} & \multicolumn{3}{c}{Strictly positive fragment} \\\hline
  $\K$ & $\PSPACE$-complete & \cite{Ladner1977} & $\K^+$ & $\PTIME$ & \cite{Bek18b, Beklemishev2021_slides} \\
  $\Kf$ and $\GL$ & $\PSPACE$-complete & \cite{Gabbay2003} & $\RC_1$ & $\PTIME$ & \cite{Dashkov2012, Beklemishev2012} \\
  $\Kfthree$ and $\GLthree$ & $\coNP$-complete & \cite{LitakWolter2005} & $\Kfthree^+$ & $\PTIME$ & \cite{Svyatlovskiy2018} \\
  $\GLPzero_\Lambda$ $(\Lambda \geq \omega)$ & $\PSPACE$-complete & \cite{Pakhomov2014} & $\RCzero_\Lambda$ & $\PTIME$ & \cite{Dashkov2012, Beklemishev2012} \\
  $\GLP$ & $\PSPACE$-complete & \cite{Shapirovsky2008} & $\RC$ & $\PTIME$ & \cite{Dashkov2012, Beklemishev2012} \\
  $\T$ & $\PSPACE$-complete & \cite{Ladner1977} & $\{\bm{e}_{refl}\}$ & $\PTIME$ & \cite{Kikot2019}\\
  $\Sfour$ & $\PSPACE$-complete & \cite{Ladner1977} & $\mathcal{E}_{qo}$ & $\PTIME$ & \cite{Kikot2019}\\
  $\Sfourthree$ & $\coNP$-complete & \cite{Gabbay2003} & $\mathcal{E}_{lin}$ & $\PTIME$ & \cite{Kikot2019} \\
  $\Sfive$ & $\coNP$-complete & \cite{Fagin1995} & $\Sfive^+$ & $\PTIME$ & \cite{Svyatlovskiy2018} \\
  $\Sfive_m$ $(m > 1)$ & $\PSPACE$-complete & \cite{Fagin1995} & $\Sfive^+_m$ & $\PTIME$ & \cite{Svyatlovskiy2018} \\
  $\QPL(\PA)$ & $\Pi_2$-complete & \cite{Vardanyan1986} & $\QRC$ & decidable & \cite{Escape}\\
  \hline
\end{tabular}
\end{table}

\section{Semantics for modal logics}\label{section:SemanticsModalLogics}

In this section, we revisit relational and arithmetical semantics for modal provability logics for the propositional and the quantified case and further state the main theorems relating the logics and their semantics.

\subsection{Relational semantics}
We briefly revisit relational semantics (also known as Kripke semantics) for classical modal logics and refer to \cite{HughesCresswell1996} for details. We start with the propositional case.

\begin{definition}[Kripke frame, Kripke model]
\label{def:Kripke_frame_model_Lb}
A Kripke frame is a tuple $\la W, R \ra$ where $W$ is a non-empty set (the worlds) and $R \subseteq W \times W$ is a binary relation (we write $w R u$ instead of $\la w, u \ra \in R$).
  A Kripke model is a Kripke frame together with a valuation function $V : \Prop \to \wp(W)$, determining which propositional symbols hold at which worlds.
\end{definition}

  The truth of a modal formula in a model $\M := \la W, R, V \ra$ at a world $w$ is defined inductively as follows:
\begin{enumerate}[label=\upshape(\roman*),ref=\thetheorem.(\roman*)]
    \item $\M, w \not\Vdash \bot$;
    \item $\M, w \Vdash p$ if and only if $w \in V(p)$, where $p \in \Prop$;
    \item $\M, w \Vdash \varphi \to \psi$ if and only if either $\M, w \not\Vdash \varphi$ or $\M, w \Vdash \psi$;
    \item $\M, w \Vdash \Box \varphi$ if and only if for every $u$ such that $w R u$ we have $\M, u \Vdash \varphi$.
  \end{enumerate}
  If $\M, w \Vdash \varphi$ for every world $w \in W$, we say that $\varphi$ is valid in $\M$ and write $\M \vDash \varphi$.

 We say that a model $\M := \la W, R, V \ra$ is finite if $W$ is finite, that it is transitive if $R$ is transitive, that it is Noetherian if every non empty subset of $W$ has an $R$-maximal element, and that it is tree-like if $R$ is acyclic and there is a root world that is connected to every other world and is not itself a successor of any world. Both $\RC_1$ and $\GL$ are complete with respect to finite transitive Noetherian tree-like models.

\begin{theorem}[\cite{Boolos1993, Dashkov2012}]\label{theorem:modalCompletenessGLandRC}

  Let $\varphi$ be an $\Lb$ formula and $\psi, \chi$ be $\Lbp$ formulas. Let quantifiers over $\M$ range over finite transitive Noetherian tree-like models with root~$r$. Then:
\begin{gather*}
  \GL \vdash \varphi \iff \left( \M, r \Vdash \varphi \text{ for any such } \M \right);\\
  \psi \vdash_{\RC_1} \chi \iff \left(\left(
    \M, r \nVdash \psi \text{ or } \M, r \Vdash \chi
  \right) \text{ for any such } \M \right).
\end{gather*}
\end{theorem}

There have been several proposals for relational semantics for quantified modal logics, from Kripke \cite{Kripke1963} to many others. Overviews can be found in \cite{HughesCresswell1996, Gabbay2009, Goldblatt2011}.
The idea of the semantics we present is to glue together several first-order models. The precise way in which to do it can vary, due to different ways to interpret the relationship between quantifiers and modal symbols, as well as with the first-order domains. In particular, if different worlds are allowed different domains, there is more than one way of relating those domains with each other. Here we describe Kripke sheaves \cite{Gabbay2009}, which for any worlds $w$ and $u$ such that $u$ is accessible from $w$ includes a compatibility function $\eta_{w, u}$ translating the domain of $w$ into the domain of $u$.

\begin{definition}[Kripke sheaf, first-order Kripke model]
\label{def:Kripke_sheaf}

  A Kripke sheaf $\Sh$ is a tuple $\la W, R,\allowbreak \{M_w\}_{w \in W}, \{\eta_{w, u}\}_{wRu} \ra$ where:
\begin{enumerate}[label=\upshape(\roman*),ref=\thetheorem.(\roman*)]
    \item $W$ is a non-empty set (the set of worlds, where individual worlds are referred to as $w, u, v$, etc);
    \item $R$ is a binary relation on $W$ (the accessibility relation);
    \item each $M_w$ is a non-empty set (the domain of the world $w$, whose elements are referred to as $d, d_0, d_1$, etc);
    \item if $wRu$, then $\eta_{w, u}$ is a function from $M_w$ to $M_u$.
  \end{enumerate}
  Furthermore, we require two conditions on the compatibility functions:

\begin{enumerate}[label=\upshape(\roman*),ref=\thetheorem.(\roman*)]\setcounter{enumi}{4}
    \item if $w R u$, $u R v$, and $w R v$, then for every $d \in M_w$ we have $\eta_{u, v}(\eta_{w,u}(d)) = \eta_{w, v}(d)$;
    \item if $w R w$ then $\eta_{w, w}$ is the identity function on $M_w$.
  \end{enumerate}

  Given a signature $\Sigma$, a first-order Kripke model $\M$ is a tuple $\la \Sh, \{I_w\}_{w \in W}, \{J_w\}_{w \in W} \ra$ where $\Sh := \la W, R,\allowbreak \{M_w\}_{w \in W},\allowbreak \{\eta_{w, u}\}_{wRu} \ra$ is a Kripke sheaf and:
  \begin{itemize}
    \item for each $w \in W$, the interpretation $I_w$ assigns an element of the domain $M_w$ to each constant $c$ of $\Sigma$, written $c^{I_w}$;
    \item $I$ is concordant: if $wRu$, then $c^{I_u} = \eta_{w, u}(c^{I_w})$ for every constant $c$ of $\Sigma$.
    \item for each $w \in W$, the interpretation $J_w$ assigns a set of tuples $S^{J_w} \subseteq \wp((M_w)^n)$ to each $n$-ary relation symbol $S$ of $\Sigma$.
  \end{itemize}
  In this case, we say that the model $\M$ is based on the sheaf $\Sh$.

  We say that a Kripke sheaf is finite if both $W$ and the domains $M_w$ for every $w \in W$ are finite, and that it has a constant domain if all the $M_w$ coincide and all the $\eta_{w, u}$ are the identity function. We say that a model is finite (respectively has a constant domain) if it is based on a finite (respectively constant domain) sheaf.

\end{definition}

In order to define the notion of satisfaction, we need one more ingredient, namely a way to interpret free variables. For this, we use assignments. Given a world $w$, a $w$-assignment $g$ is a function form the set of variables to the domain $M_w$.
Any $w$-assignment can be seen as a $u$-assignment as long as $w R u$, by composing it with $\eta_{w, u}$ on the left. We write $g^u$ to shorten $\eta_{w, u} \circ g$ when $w$ is clear from context.
A $w$-assignment $g$ is extended to terms by defining $g(c) := c^{I_w}$ for any constant $c$. Note that this meshes nicely with the concordance restriction of the model: for any term $t$, if $wRu$ then $g^u(t) = \eta_{w, u}(g(t))$.
Two $w$-assignments $g$ and $h$ are $x$-alternative, written $g \xaltern{x} h$, if they coincide on all variables other than $x$.

We are now ready to define the notion of truth.

\begin{definition}[Truth]
  Let $\M = \la W, R, \{M_w\}_{w \in W},\allowbreak \{\eta_{w, u}\}_{wRu},\allowbreak \{I_w\}_{w \in W}, \{J_w\}_{w \in W} \ra$ be a first-order Kripke model in some signature $\Sigma$, and let $w \in W$ be a world, $g$ be a $w$-assignment, $S$ be an $n$-ary relation symbol, and $\varphi, \psi$ be $\Lbf$-formulas in the language of $\Sigma$.

  We define $\M, w \Vdash^g \varphi$ ($\varphi$ is true at $w$ under $g$) by induction on $\varphi$ as follows:
  \begin{itemize}
    \item $\M, w \not\Vdash^g \bot$;

    \item $\M, w \Vdash^g S(t_0, \hdots, t_{n-1})$ if and only if $\la g(t_0), \hdots, g(t_{n-1}) \ra \in S^{J_w}$;

    \item $\M, w \Vdash^g \varphi \to \psi$ if and only if either $\M, w \not\Vdash^g \varphi$ or $\M, w \Vdash^g \psi$;

    \item $\M, w \Vdash^g \Forall{x} \varphi$ if and only if for all $w$-assignments $h$ such that $h \xaltern{x} g$ we have $\M, w \Vdash^{h} \varphi$;

    \item $\M, w \Vdash^g \Box \varphi$ if and only if for every $u \in W$ such that $w R u$ we have $\M, u \Vdash^{g^u} \varphi$.
  \end{itemize}

  If $\M, w \Vdash^g \varphi$ for every $w$-assignment $g$ we say that $\varphi$ is satisfied at $w$ and write $\M, w \Vdash \varphi$. If $\varphi$ is satisfied at every world $w$, we say that $\varphi$ is valid in $\M$ and write $\M \vDash \varphi$. Finally, if $\varphi$ is valid in every model $\M$ based on a given sheaf $\Sh$, we say that $\varphi$ is valid in $\Sh$ and write $\Sh \vDash \varphi$.
\end{definition}

The soundness of a propositional modal logic can be leveraged to obtain soundness with respect to Kripke sheaves for its quantified counterpart.

\begin{theorem}[Soundness for $\Q L$]
\label{thm:QSsoundness}
Let $\F$ be a Kripke frame for $\Lb$ as described in Definition~\ref{def:Kripke_frame_model_Lb}. If $\F$ satisfies every theorem of a propositional modal logic $L$, then any sheaf based on $\F$ (i.e., with the same set of worlds and accessibility relation) satisfies every theorem of $\Q L$.
\end{theorem}
\begin{proof}
  By induction on a $\Q L$ proof of a $\Lbf$ formula. The proof is analogous to the one for the non-sheaf quantified Kripke semantics described in \cite{HughesCresswell1996}. See details in \cite[Theorem~2.5.4]{BorgesPhD}.
\end{proof}

As always, completeness has to be independently established for each different logic. Modal completeness for $\QRC$ is proved in \cite{QRC1}, and this result is strengthened to constant-domain completeness in \cite{Escape}.

\begin{theorem}[\protect{\cite[Theorem~5.11]{Escape}}]\label{theorem:QRCcomplete}
  If $\varphi \nvdash_{\QRC} \psi$, then there are a finite Kripke model $\M$ with a constant finite domain where $R$ is irreflexive and transitive, a world $w \in \M$, and a $w$-assignment $g$ such that $\M, w \Vdash^g \varphi$ and $\M, w \nVdash^g \psi$.
\end{theorem}

\subsection{Arithmetical semantics}

We refer to \cite{HajekPudlak:1993:Metamathematics, Troelstra1973} for definitions of $\PA$ and $\HA$, which are Robinson's arithmetic with full induction over classical or constructive logic, respectively. As always, $\Delta_0$ refers to formulas where all quantifiers are bounded by some term in our language and we set $\Pi_0 := \Sigma_0 := \Delta_0$. The class $\Sigma_{n+1}$ arises from existentially quantifying $\Pi_{n}$ formulas and the class $\Pi_{n+1}$ arises from universally quantifying $\Sigma_{n}$ formulas. Given a class $\Gamma$ (for example $\Sigma_1$) and a theory $T$, we say that a formula $\varphi$ is ``$\Gamma$ over $T$'' if $T$ proves that $\varphi$ is equivalent to $\psi$, for some formula $\psi \in \Gamma$.

The theory $\EA$ is as $\PA$ with the induction axiom restricted to $\Delta_0$ formulas. 
The intuitionistic version of $\EA$ is denoted by $\iEA$ and it only differs from $\EA$ in that it is defined over intuitionistic predicate logic instead of classical predicate logic. A theory is said to be $\Gamma$-sound if all formulas in $\Gamma$ provable in the theory are valid in the standard model $\N$.

Modal logics are related to arithmetic via Gödel numbering and via interpreting the $\Box$ modality as formalized provability. For an arithmetical formula $\varphi$, we denote its Gödel number by $\ulcorner \varphi \urcorner$. We suggestively denote the arithmetical formalized provability predicate by $\Box_T$ so that we have $T\vdash \varphi$ if and only if $\mathbb N \vDash \Box_T(\ulcorner \varphi \urcorner)$. Strictly speaking, an axiomatization $\tau$ of a theory $T$ induces a provability predicate $\Box_\tau$, but we shall often write $\Box_T$ when the intended axiomatization $\tau$ for $T$ is clear from context. We define $\Diamond_\tau \varphi$ as $\neg \Box_\tau \neg \varphi$.

Propositional variables in modal logics are taken to represent placeholders for arbitrary arithmetical formulas. To formalize this idea, we use so-called realizations.

\begin{definition}[Arithmetical realization for $\Lb$, $\cdot\pfs$, $\cdot\pfsT$]
\label{def:propositional:realization}
  An arithmetical realization $\cdot\pfs$ is a function from $\Prop$ to sentences in the language of arithmetic. Given a c.e.~theory $T$ in the language of arithmetic and an arithmetical realization $\cdot\pfs$, we define $\cdot\pfsT$ on modal formulas as follows:
  \begin{itemize}
    \item $\bot\pfsT := (0 = 1)$;
    \item $p\pfsT := p\pfs$, where $p \in \Prop$;
    \item $(\varphi \to \psi)\pfsT := \varphi\pfsT \to \psi\pfsT$;
    \item $(\Box \varphi)\pfsT := \Box_T \varphi \pfsT$.
  \end{itemize}
\end{definition}
In the case of constructive theories such as $\HA$, the definition is adapted in the usual way so that $\cdot\pfsT$ commutes with the connectives $\land$ and $\lor$.

We can now define the provability logic $\PL(T)$ and the strictly positive provability logic $\PL^+(T)$ of a theory $T$ in the language as:
\begin{equation*}
  \PL(T) :=
    \{
      \varphi \in \Lb
      \mid
      \text{for any }\cdot\pfs, \text{ we have } T \vdash \varphi\pfsT
    \},
\end{equation*}
\begin{equation*}
  \PL^+(T) :=
    \{
      \la \varphi, \psi \ra \in \Lbp \times \Lbp
      \mid
      \text{for any }\cdot\pfs, \text{ we have } T \vdash \varphi\pfsT \to \psi\pfsT
    \}.
\end{equation*}
We observe that, in virtue of their definitions, $\PLsp(T) = (\PL (T))^+$.

We can now state that $\PL(T) = \GL$ for a large class of theories $T$. The version for $\PA$ is known as Solovay's completeness theorem.

\begin{theorem}[\cite{Solovay:1976, JonghJumeletMontagna:1991}]
\label{thm:Solovay}
  Let $T$ be a c.e.~$\Sigma_1$-sound extension of $\EA$ and let $\varphi \in \Lb$. Then:
  \begin{equation*}
    \GL \vdash \varphi
    \iff
    \text{for any }\cdot\pfsT, \text{we have }T \vdash \varphi\pfsT
    .
  \end{equation*}
\end{theorem}

Strictly positive propositional logics have an alternative arithmetical interpretation where propositional variables are not mapped to sentences but rather to formulas axiomatizing a theory.

\begin{definition}[Infinitary arithmetical realization for $\Lbp$, $\cdot\is$, $\cdot\isT$]
\label{def:propositional:sp:realization}
Let $\cdot\is$ be a function from propositional symbols to axiomatizations of theories of arithmetic. Given a base axiomatization $\tau$, we extend $\cdot\is$ to formulas of $\Lbp$ as follows:
  \begin{itemize}
    \item $\top\isT := \tau(u)$;
    \item $p\isT : = p\is \lor \tau(u)$;
    \item $(\varphi \land \psi)\isT := \varphi\isT \lor \psi\isT$;
    \item $(\Diamond \varphi)\isT := \tau(u) \lor (u = \gnum{\Diamond_{\varphi\isT} \top})$.
  \end{itemize}
\end{definition}

Given these infinitary realizations, we define the reflection logic of a theory:
  \begin{equation}\label{equation:infinitaryRC}
  \IPL(T) := \{
    \la \varphi, \psi \ra \in \Lbp \times \Lbp \mid
    \text{for any }\cdot\is, \text{ we have }
    T \vdash \Forall{\theta}  (\Box_{\psi\isT} \theta \to \Box_{\varphi\isT} \theta)
  \}.
  \end{equation}
Clearly, infinitary realizations extend finitary ones, for given a finitary realization $\cdot\pfs$ we can define $p\is := (u = \gnum{p\pfs})$. Then $p\isT$ axiomatizes the base theory extended by the single axiom $p\pfs$. Using this observation it is not hard to see that $\IPL(\PA) = \RC_1$ and in Section \ref{section:PLHA+} we show that $\IPL(\HA) = \RC_1$ as well.

We extend the notion of (finitary) arithmetical realization to $\Lbf$ in the natural way.
To simplify matters, in the sections dealing with both the quantified modal language and the language of arithmetic, we always use variables named $x_i$ (for some $i$) in the modal setting and variables named $y_i$ or $z_i$ in the arithmetical setting.

\begin{definition}[Arithmetical realization for $\Lbf$ and $\Lp$, $\cdot\fs$, $\cdot\fsT$]
\label{def:QML:realization}
  An arithmetical realization $\cdot\fs$ is a function from relation symbols applied to a tuple of $n$ variables (corresponding to the symbol's arity) to formulas in the language of arithmetic with $n$ free variables. We require that a variable $x_i$ in the modal language be translated to $y_i$ in the arithmetical setting and that a constant $c_i$ in the modal language be translated to $z_i$ in the arithmetical setting.

  Given a c.e.~theory $T$ extending $\iEA$, and an arithmetical realization $\cdot\fs$, we define $\cdot\fsT$ on quantified modal formulas as follows (extendable to the constructive setting as before):
\begin{itemize}
  \item $\bot\fsT := (0 = 1)$;
  \item $S(\vec t)\fsT := S(\vec t)\fs$;
  \item $(\varphi \to \psi)\fsT := \varphi\fsT \to \psi\fsT$;
  \item $(\Box \varphi)\fsT := \Box_T \varphi\fsT$;
  \item $(\Forall{x_k} \varphi)\fsT := \Forall{y_k} \varphi\fsT$.
\end{itemize}
This definition can be restricted to $\Lp$ in the obvious way.
\end{definition}

After Solovay's completeness theorem (Theorem~\ref{thm:Solovay}), it was natural to ask whether one could find a nice characterization for the counterpart of $\PL(\PA)$ in the quantified modal language, defined as follows (see Boolos \cite{Boolos1993}).
\begin{definition}[Quantified provability logic, $\QPL$]
\label{def:QPL}
If $T$ is a c.e.~extension of $\iEA$, the quantified provability logic of $T$ is defined as:\footnote{Recall that we assume all modal variables are of the form $x_i$ and that these are translated to the arithmetical variables $y_i$. We also assume that all modal constants are translated to some $z_i$, but note that constants are not included in the language $\Lbf$, only in $\Lp$. With $\Forall{\vec y} \varphi\fsT$ we mean to quantify over all the free variables of $\varphi\fsT$, and similarly in the $\QPLsp(T)$ case.}
\begin{equation*}
  \QPL(T) :=
    \{
      \varphi \in \Lbf
      \mid
      \text{for any }\cdot\fs, \text{ we have } T \vdash \Forall{\vec y} \varphi\fsT
    \}
    ,
\end{equation*}
\begin{equation*}
  \QPLsp(T) :=
    \{
      \la \varphi, \psi \ra \in \Lp \times \Lp
      \mid
      \text{for any }\cdot\fs, \text{ we have } T \vdash \Forall{\vec y, \vec z} (\varphi\fsT \to \psi\fsT)
    \}
    .
\end{equation*}
\end{definition}

In general, $\QPL(T)$ can be different depending on the chosen axiomatization $\tau$ for $T$ (see \cite{Artemov1986, Kurahashi2013, Kurahashi2022}), but here we assume that an axiomatization is fixed for each relevant theory for simplicity.
It is not exactly the case that $\QPLsp(T) = (\QPL(T))^+$ due to the difference in weather constants are allowed in the language. However, for a fixed signature $\Sigma$ without constants, we indeed have $\QPLsp(T) = (\QPL(T))^+$.

The natural candidate for a quantified provability logic for $\PA$ was $\QGL$, but although $\QGL \subseteq \QPL(\PA)$, Vardanyan \cite{Vardanyan1986} and McGee \cite{McGee1985} independently showed that $\QPL(\PA)$ is as complex as it can possibly be: $\Pi_2$-complete. However, in \cite{Escape} it was proven that the strictly positive fragment of this $\Pi_2$-complete logic is given by $\QRC$. In other words, for the corresponding signatures we have the following result.

\begin{theorem}[\protect{\cite[Theorem~9.1]{Escape}}]\label{theorem:QPLspEqualsQRCOne}
  $\QPLsp(\PA) = \QRC$.
\end{theorem}

In Section \ref{section:QPLplusHA} we prove that likewise  $\QPLsp(\HA) = \QRC$.

The strictly positive fragment $\Lp$ again allows for an infinitary arithmetical interpretation. In that case, a predicate $S(\vec x, \vec c)$ in the signature $\Sigma$ is mapped to a formula $S\is(\vec y, \vec z)$ that axiomatizes a collection of theories parametrized by $\vec y, \vec z$.

\begin{definition}[Infinitary arithmetical realization for $\Lp$, $\cdot\is$, $\cdot\isT$]
\label{def:isT}

Let $\cdot\is$ be a realization such that, for a given $n$-ary predicate $S$ and terms $\vec t$, $S(\vec t)\is$ is an $(n+1)$-ary $\Sigma_1$ arithmetical formula where modal variables $x_k$ are interpreted as $y_k$ and modal constants $c_k$ are interpreted as $z_k$. Let $\tau$ be a $\Sigma_1$ axiomatization of some base theory $T$. We extend this realization to $\QRC$ formulas as follows:
  \begin{itemize}
    \item $\top\isT := \tau(u)$;
    \item $S(\vec{x}, \vec{c})\isT : = S(\vec{x},\vec{c})\is \lor \tau(u)$;
    \item $(\psi \land \delta)\isT := \psi\isT \lor \delta\isT$;
    \item $(\lozenge \psi)\isT := \tau(u) \lor (u = \gnum{\Diamond_{\psi\isT} \top})$;
    \item $(\Forall{x_i} \psi)\isT := \Exists{y_i} \psi\isT$.
  \end{itemize}
\end{definition}

Using the infinitary semantics, we can define the quantified reflection logic $\IQPL(T)$ of a theory $T$ in analogy with \eqref{equation:infinitaryRC} as:
  \begin{equation*}
  \IQPL(T) := \{
    \la \varphi, \psi \ra \in \Lp \times \Lp \mid
    \text{for any }\cdot\is, \
    T \vdash \Forall{\theta} \Forall{\vec{y}, \vec{z}} (\Box_{\psi\isT} \theta \to \Box_{\varphi\isT} \theta)
  \}.
  \end{equation*}
  In \cite[Theorems~7.2 and 8.9]{Escape} it is shown that $\IQPL(\PA) = \QRC$ and in Section~\ref{section:QPLplusHA} we prove that furthermore $\IQPL(\HA) = \QRC$.

\section{The arithmetical hierarchy in \texorpdfstring{$\HA$}{HA}}\label{section:hierarchyInHA}

Heyting Arithmetic does not permit a simple treatment of the arithmetical hierarchy (see e.g.~\cite{FujiwaraKurahashi2023}). In this section we collect some basic observations on approximations to classical reasoning regarding formula complexity, which we use later in the paper.

\subsection{Closure and semi-closure under double negations}

Using standard coding techniques, it is possible to see that $\HA$ proves that  $\Sigma_1$ sentences are closed under conjunction, and this fact is formalizable in $\HA$.
We furthermore recall that for any arithmetical formula $\varphi$ we have:
\begin{equation}
\label{equation:forallNotEquivNotExists}
  \HA \vdash
    \Forall x \neg \varphi \leftrightarrow \neg \Exists x \varphi
    .
\end{equation}
This equivalence can be contrasted with $\HA \nvdash \Exists x \neg \varphi \leftrightarrow \neg \Forall x \varphi$.

Additionally, $\HA$ proves excluded middle for decidable formulas. Consequently, if $\delta$ is a decidable formula, then:
\begin{equation}
\label{equation:excludedMiddleForDecidable}
  \HA \vdash \neg \neg \delta \leftrightarrow \delta.
\end{equation}
In fact, this holds for arithmetical formulas $\delta$ of complexity $\Delta_1$.
However, excluded middle fails to hold for genuine $\Pi_1$ sentences $\pi$, but double negation elimination is still possible:
\begin{equation*}
  \HA \vdash \neg \neg \pi \leftrightarrow \pi.
\end{equation*}
The same equivalence does not hold for genuine $\Sigma_1$ sentences $\sigma$:
\begin{equation*}
  \HA \nvdash \neg \neg \sigma \leftrightarrow \sigma.
\end{equation*}
Even though double negation elimination fails in general for $\Sigma_1$ sentences, we do have a certain formalized version. In Lemma~\ref{theorem:extendedDoubleNegation} we will see that $\HA \vdash \Box_\HA \sigma \leftrightarrow \Box_\HA \neg \neg \sigma$ for $\Sigma_1$ sentences $\sigma$. This is due to a certain amount of conservatism between $\HA$ and $\PA$ that shall be discussed in the next subsection.

The class of $\Sigma_2$ formulas is particularly important for our study. In pure constructive predicate logic, $\Sigma_2$ formulas are provably closed under conjunctions but not under disjunctions. However, in $\HA$ we do have the latter.

\begin{lemma}\label{theorem:SigmaTwoHaClosedUnderDisjunctions}
For each $\varphi_0, \varphi_1 \in \Sigma_2$, there is $\varphi\in \Sigma_2$ such that:
\begin{equation*}
  \HA \vdash \varphi_0 \lor \varphi_1 \leftrightarrow \varphi
  .
\end{equation*}
\end{lemma}

\begin{proof}
Let $\varphi_0 = \exists s\, \forall t\, \tilde \varphi_0 (s, t)$ and $\varphi_1 = \exists u\, \forall v\, \tilde \varphi_1 (u, v)$,
with $\tilde \varphi_1 (u, v), \tilde \varphi_0 (s, t)\in \Delta_0$. We can take:
\begin{equation*}
  \varphi := \Exists x \Forall y (
    ((x)_0 = 0 \land \tilde \varphi_0 ((x)_1, y))
    \lor
    ((x)_0 = 1 \land \tilde \varphi_1 ((x)_1, y))
  )
  ,
\end{equation*}
where $x$ is being interpreted as a pair, of which $(x)_0$ and $(x)_1$ are the first and second projections, respectively.
\end{proof}

\subsection{Equi-consistency results in \texorpdfstring{$\HA$}{HA}}

When moving between $\PA$ and $\HA$, one conservativity direction is trivial and formalizable in $\HA$:
\begin{equation}\label{equation:HAincludedInPAFormalized}
  \HA \vdash \Forall \varphi (\Box_\HA \varphi \to \Box_\PA \varphi) 
  .
\end{equation}
The other direction is known to hold provably in $\HA$ for $\Pi_2$ sentences (see \cite{Friedman1978}),
both in the non-formalized setting and in the formalized setting, so we have an equivalence:
\begin{equation}\label{equation:provableInHAiffPA}
  \HA \vdash \Forall{\pi \in \Pi_2}
    (\Box_\PA \pi \leftrightarrow \Box_\HA \pi)
  .
\end{equation}
With this conservativity result, we obtain a certain amount of classical behavior for $\Pi_2$ sentences in a formalized setting.
In particular, combining the above we have a $\HA$ provable semi-closure of a certain kind of double negation.

\begin{lemma}\label{theorem:extendedDoubleNegation}
\hfill

  \begin{enumerate}
    \item
      $\HA \vdash
        \Forall{\sigma \in \Sigma_1} (
          \Box_\HA \neg \neg \sigma \leftrightarrow \Box_\HA \sigma
        )
      $;

    \item
      $\HA \vdash
        \Forall{\sigma \in \Sigma_1} (
          \Box_\HA \Forall x \neg \neg \sigma \leftrightarrow \Box_\HA \Forall x \sigma
        )
      $.
\end{enumerate}
\end{lemma}
\begin{proof}
  Clearly, the first statement follows from the second, and so we only focus on the latter. Reason in $\HA$ and fix an arbitrary $\sigma \in \Sigma_1$. Since clearly $\sigma \to \neg \neg \sigma$, we focus on the other implication. Consider:
\begin{align*}
  \Box_\HA \Forall x \neg \neg \sigma
    &\to \Box_\PA \Forall x \neg \neg \sigma
      &&\text{by \eqref{equation:HAincludedInPAFormalized}}\\
    &\to \Box_\PA \Forall x \sigma
      &&\\
    &\to \Box_\HA \Forall x \sigma
      &&\text{by \eqref{equation:provableInHAiffPA}.}
\end{align*}
\end{proof}

Combining some of the above observations, it is easy to see that the negation of a $\Pi_1$ sentence is equivalent to the double negation of a $\Sigma_1$ sentence over $\HA$.
\begin{lemma}\label{lemma:negationOfPiEquivToNotNotSigma}
  Given a $\Pi_1$ sentence $\pi$, there is a $\Sigma_1$ sentence $\sigma$ such that:
  \begin{equation*}
    \HA \vdash \neg \pi \leftrightarrow \neg \neg \sigma
    .
  \end{equation*}
\end{lemma}
\begin{proof}
  Let $\pi := \Forall x \delta$ where $\delta \in \Delta_0$.
  Consider:
  \begin{align*}
    \HA \vdash \neg \Forall x \delta
      &\leftrightarrow \neg \Forall x \neg \neg \delta
        &&\\
      &\leftrightarrow \neg \neg \Exists x \neg \delta
        &&\text{by \eqref{equation:forallNotEquivNotExists}.}
  \end{align*}
  Then $\sigma := \Exists x \neg \delta$ is the required $\Sigma_1$ sentence.
\end{proof}

We can now prove a dual version of \eqref{equation:provableInHAiffPA}. 

\begin{lemma}
\label{theorem:DiamondSigmaTwoConservativity}
\hfill

  \begin{enumerate}
    \item
      $\HA \vdash
        \Forall{\varphi \in \Sigma_2} (
          \Box_{\HA} \neg \varphi \leftrightarrow \Box_{\PA} \neg \varphi
        )
      $;
    \item
      $\HA \vdash
        \Forall{\varphi \in \Sigma_2} (
          \Diamond_{\HA} \varphi \leftrightarrow \Diamond_{\PA} \varphi
        )
      $.
  \end{enumerate}
\end{lemma}
\begin{proof}
  The second statement easily follows from the first. Reason in $\HA$ and fix an arbitrary $\varphi \in \Sigma_2$. We can write $\varphi$ as $\Exists x \pi$ for some $\pi \in \Pi_1$. Let $\sigma$ be the $\Sigma_1$ sentence given by Lemma~\eqref{lemma:negationOfPiEquivToNotNotSigma} for our fixed $\pi$ such that:
\begin{equation}
\label{equation:PversusS}
  \Box_{\HA}(\neg \pi \leftrightarrow \neg \neg \sigma)
  .
\end{equation}
Now consider:
\begin{align*}
\Box_\HA \neg \varphi
    &\leftrightarrow \Box_\HA \neg \Exists x \pi
      &&\\
    &\leftrightarrow \Box_\HA \Forall x \neg \pi
      &&\text{by \eqref{equation:forallNotEquivNotExists}}\\
    &\leftrightarrow \Box_\HA \Forall x \neg \neg \sigma
      &&\text{by  \eqref{equation:PversusS}}\\
    &\leftrightarrow \Box_\HA \Forall x \sigma
      &&\text{by Lemma~\ref{theorem:extendedDoubleNegation}}\\
    &\leftrightarrow \Box_\PA \Forall x \sigma
      &&\text{by \eqref{equation:provableInHAiffPA}}\\
    &\leftrightarrow \Box_\PA \neg \neg \Forall x \sigma
      &&\\
    &\leftrightarrow \Box_\PA \neg \Exists x \neg \sigma
      &&\\
    &\leftrightarrow \Box_\PA \neg \Exists x \pi
      &&\text{by \eqref{equation:PversusS}}\\
    &\leftrightarrow \Box_\PA \neg \varphi.
      &&
\end{align*}
\end{proof}

The above lemma is reminiscent of the following result by Kreisel.
\begin{theorem}[Kreisel~\cite{Kreisel1951}]
  For every formula $\varphi$ in prenex normal form we have:
  \begin{equation*}
    \HA \vdash \Box_\PA \neg \varphi \leftrightarrow \Box_\HA \neg \varphi
    .
  \end{equation*}
\end{theorem}
\begin{proof}
The right to left implication is is clear, so we focus on the other direction. Reason in $\HA$ and assume:
\begin{equation*}
  \Box_\PA \neg \forall x_1 \exists y_1 \forall x_2 \exists y_2 \cdots \delta
  ,
\end{equation*}
taking $\varphi$ as $\forall x_1 \exists y_1 \forall x_2 \exists y_2 \cdots \delta$ for some $\Delta_0$-formula $\delta$, where the first universal quantification may be vacuous. By the Kuroda translation \cite{Kuroda1951}, we deduce:
  \begin{equation*}
    \Box_\HA \neg \forall x_1 \neg \neg \exists y_1 \forall x_2 \neg \neg \exists y_2 \cdots \delta
    .
  \end{equation*}
  Since $\HA \vdash \forall x_1 \exists y_1 \forall x_2 \exists y_2 \cdots \delta \to \forall x_1 \neg \neg \exists y_1 \forall x_2 \neg \neg \exists y_2 \cdots \delta$, we obtain:
  \begin{equation*}
    \Box_\HA \neg \forall x_1  \exists y_1 \forall x_2  \exists y_2 \cdots \delta
    ,
  \end{equation*}
that is, $\Box_\HA \neg \varphi$, as desired.
\end{proof}

From Kreisel's observation we directly obtain the following corollary.

\begin{corollary}
For every formula $\varphi$ in prenex normal form we have:
\begin{equation*}
  \HA \vdash \Diamond_\PA  \varphi \leftrightarrow \Diamond_\HA  \varphi.
\end{equation*}
\end{corollary}

\section{Reflections on arithmetical completeness}\label{section:Solovay}

In this section we briefly revisit the ingredients of Solovay's arithmetical completeness proof for $\GL$.
We have a particular interest in the complexity of the formulas occurring in that proof so that we can invoke conservativity of $\PA$ over $\HA$. We first observe that the complexity of the arithmetical realizations of propositional strictly positive formulas can be kept low.

\begin{lemma}\label{lemma:boundedComplexityRealisations}
  For every realization $\cdot\pfs : \Prop \to \Sigma_2$, every c.e.~theory $T$ and every p.s.p.~formula $\varphi$ we have that $\varphi\pfsT$ is $\Sigma_2$ over $\HA$, and this is provable in $\HA$.
\end{lemma}
\begin{proof}
By an easy induction on the complexity of $\varphi$. The only interesting case is the complexity of $\Diamond \psi$, which is taken care of by \eqref{equation:forallNotEquivNotExists} since the negation of a $\Sigma_1$ sentence is $\Pi_1$, whence also $\Sigma_2$.
\end{proof}

This simple observation can be used to see that extending a realization using $\HA$ provability or $\PA$ provability yields $\HA$-equivalent formulas, provably in $\HA$.

\begin{lemma}\label{theorem:HARealisationEquivalentToPA}
  For every realization $\cdot\pfs : \Prop \to \Sigma_2$ and every p.s.p.~formula $\varphi$ we have:
  \begin{equation*}
    \HA \vdash \varphi\pfsHA \leftrightarrow \varphi\pfsPA.
  \end{equation*}
\end{lemma}

\begin{proof}
By an easy induction on $\varphi$ where the $\Diamond$ case is taken care of by Lemma~\ref{theorem:DiamondSigmaTwoConservativity}.
\end{proof}

Solovay proved the arithmetical completeness theorem for $\GL$ by embedding Kripke frames into arithmetic. Given a Kripke frame with worlds $1, \ldots, N$, a new world $0$ is added such that it can see all the other worlds. Then, each $i \in \{0, 1, \hdots, N\}$ is assigned an arithmetical sentence $\lambda_i$ such that each $\lambda_i$ is independent from $\PA$ and that these $\lambda_i$ sentences mimic the accessibility relations between their respective worlds as described by the following lemma.
Then the realization $\cdot\pfs$ of Solovay's proof is defined as $p\pfs := \bigvee_{i \Vdash p}\lambda_i$.

\begin{lemma}[Embedding \cite{Boolos1993}]
\label{lem:embedding}
\
\begin{enumerate}[ref=\thetheorem.\arabic*]
  \item\label{embedding:SomeOfAll}
    $\PA \vdash \bigvee_{i \leq N} \lambda_i$;

  \item\label{embedding:OneOfAll}
    $\PA \vdash \lambda_i \to \bigwedge_{j \leq N, j \neq i} \neg \lambda_j$, for $i \leq N$;

  \item\label{embedding:Diamond}
    $\PA \vdash \lambda_i \to \bigwedge_{j \leq N, iRj} \Diamond_\PA \lambda_j$, for $i \leq N$;

  \item\label{embedding:Box}
    $\PA \vdash \lambda_i \to  \Box_\PA \bigvee_{j \leq N, iRj} \lambda_j$, for $0 < i \leq N$;

  \item\label{embedding:Lambda0istrue}
    $\N \vDash \lambda_0$.
\end{enumerate}
\end{lemma}

The sentence $\lambda_i$ states that the limit of a so-called Solovay function $f : \mathbb N \to \{0, 1, \hdots, N\}$ is $i$. Each sentence is thus \emph{a priori} of complexity $\Sigma_2$. However, by definition, the Solovay function can only move along $R$-transitions in the frame, so we can write $\lambda_i$ as a conjunction of $\Pi_1$ and $\Sigma_1$ sentences as follows:
\begin{equation}\label{equation:SolovayFunctionLow}
  \lambda_i =
    (\Exists{x} f(x) = i)
    \land
    \Forall{x, y} (x \leq y \to (f(x) = i \to f(y) = i))
  .
\end{equation}

We denote the class of formulas built of conjunctions and disjunctions of $\Pi_1$ and $\Sigma_1$ formulas by $\DCone$.
Since in Solovay's proof $p\pfs := \bigvee_{i\Vdash p}\lambda_i$, it is clear that the realization satisfies $\cdot\pfs : \Prop \to \DCone$. In fact, $\varphi\pfsT$ still has complexity $\DCone$, which can be proved by a simple induction on $\varphi$.

\begin{remark}\label{remark:substitutionsLow}
  For any p.s.p.~formula $\varphi$ we have $\varphi\pfsT \in \DCone$. 
\end{remark}

Note that $\DCone \subseteq \Sigma_2$ over $\HA$ by Lemma~\ref{theorem:SigmaTwoHaClosedUnderDisjunctions}. As a direct consequence of the previous observations we obtain the following result.

\begin{corollary}\label{cor:notAHAeqnotAPA}
  Let $\cdot\pfs$ be a realization obtained via Solovay's embedding of a finite transitive Noetherian tree-like model. Then for any p.s.p.~formula $\varphi$:
\begin{equation*}
  \HA \vdash \neg \varphi\pfsHA \iff \PA \vdash \neg \varphi\pfsPA
  .
\end{equation*}
Moreover, this is formalizable in $\HA$.%
\end{corollary}
\begin{proof}
By Remark~\ref{remark:substitutionsLow}, Lemma~\ref{theorem:HARealisationEquivalentToPA}, and Lemma~\ref{theorem:DiamondSigmaTwoConservativity}.
\end{proof}

Note that we cannot prove a version of Corollary~\ref{cor:notAHAeqnotAPA} without the negation: if $\PA \vdash p\pfsPA$ implied $\HA \vdash p\pfsHA$, then by the disjunction property we would have $\HA \vdash \lambda_i$ for some $i$, which would contradict the independence of $\lambda_i$ from $\PA$.

\section{The strictly positive fragment of \texorpdfstring{$\PL(\HA)$}{PL(HA)}}
\label{section:PLHA+}

In this section we show that the strictly positive provability logic of $\HA$ and the reflection logic of $\HA$ both coincide with $\RC_1$.

\begin{theorem}\label{theorem:PLpOfHAisRC}
  Let $\varphi$ and $\psi$ be p.s.p.~formulas. We have $\varphi \vdash_{\RC_1} \psi$ if and only if for all realizations $\cdot\pfs$ we have $\HA \vdash \varphi\pfsHA \to \psi\pfsHA$.
\end{theorem}
\begin{proof}
  Soundness is easy given the derivability conditions of $\Box_\HA$ (see \cite{VisserZoethout2019}) and it was already incidentally observed in \cite[Theorem~5.4.3]{BorgesPhD}.
  For completeness, reason as follows:
  \begin{align*}
    \varphi \nvdash_{\RC_1} \psi
      &\implies \text{ there is } \cdot\pfs \text{ s.t. }
        \PA \nvdash \varphi\pfsPA \to \psi\pfsPA
        &&\text{by Theorem~\ref{thm:Solovay}}
        \\
      &\implies \text{ there is } \cdot\pfs \text{ s.t. }
        \PA \nvdash \varphi\pfsHA \to \psi\pfsHA
        &&\text{by Lemma~\ref{theorem:HARealisationEquivalentToPA}}
        \\
      &\implies \text{ there is } \cdot\pfs \text{ s.t. }
        \HA \nvdash \varphi\pfsHA \to \psi\pfsHA
        &&\text{by \eqref{equation:HAincludedInPAFormalized}.}
  \end{align*}
\end{proof}

We note that there is some room for improvement and generalization, arithmetically speaking. Can this be exploited to get a slightly larger fragment (disjunctions seem to come for free, and boxes with a bit of extra work)?
 
Recall that infinitary realizations extend finitary ones in the following sense. Given a finitary realization $\cdot\pfs$, we can define $p\is := (u = \gnum{p\pfs})$. Then $p\isT$ axiomatizes the base theory extended by the single axiom $p\pfs$. This infinitary realization behaves exactly like its finitary counterpart, as the following lemma illustrates.

\begin{lemma}\label{lemma:WDTST} 
  Let $\varphi$ be a p.s.p.~formula, $\cdot\pfs$ be a finitary realization and $\cdot\is$ be its finitary counterpart as described above. Then:
  \begin{equation*}
    \HA \vdash \Forall{\theta} (\Box_{\varphi\isHA} \theta \leftrightarrow \Box_\HA (\varphi\pfsHA \to \theta))
    .
  \end{equation*}
\end{lemma}
\begin{proof}
  By an easy external induction on $\varphi$.
\end{proof}

Using this the reflection logic of $\HA$ can be obtained as follows (recall \eqref{equation:infinitaryRC}).

\begin{theorem}\label{theorem:RefLogicHAisRCOne}
$\IPL(\HA) = \RC_1$.
\end{theorem}
\begin{proof}
  The soundness, that $\RC_1 \subseteq \IPL(\HA)$, follows from a straightforward induction and was already incidentally observed in \cite[Theorem~10.3]{Escape}.
  For the other direction, assume that $\varphi \nvdash_{\RC_1} \psi$. With Theorem~\ref{theorem:PLpOfHAisRC} we find $\cdot\pfs$ such that $\HA \nvdash \varphi\pfsHA \to \psi\pfsHA$. We now set $p\is := (u = \gnum{p\pfs})$. We have $\HA \vdash \Box_{\psi\isHA} \psi\pfsHA$ by Lemma~\ref{lemma:WDTST}. We now check that $\HA \nvdash \Box_{\varphi\isHA} \psi\pfsHA$ with the following reasoning:
  \begin{align*}
    \HA \vdash \Box_{\varphi\isHA} \psi\pfsHA
      &\implies \HA \vdash \Box_{\HA} (\varphi\pfsHA \to \psi\pfsHA)
        &&\text{by Lemma~\ref{lemma:WDTST}}
        \\
      &\implies \N \vDash \Box_{\HA} (\varphi\pfsHA \to \psi\pfsHA)
        &&\text{by $\Sigma_1$-soundness}
        \\
      &\implies \HA \vdash \varphi\pfsHA \to \psi\pfsHA.
        &&
  \end{align*}
  Since $\HA \nvdash \varphi\pfsHA \to \psi\pfsHA$ by the choice of $\cdot\pfs$, we conclude that $\HA \nvdash \Box_{\varphi\isHA} \psi\pfsHA$.
  Consequently, $\HA \nvdash \Forall{\theta} (\Box_{\psi\isHA} \theta \to \Box_{\varphi\isHA} \theta)$, as was to be shown.
\end{proof}

\section{The strictly positive fragment of \texorpdfstring{$\QPL(\HA)$}{QPL(HA)}}
\label{section:QPLplusHA}

In this section we prove that $\QPLsp(\HA) = \IQPL(\HA) = \QRC$. The proofs are essentially the same as the ones in the propositional case as stated in Theorems~\ref{theorem:PLpOfHAisRC} and \ref{theorem:RefLogicHAisRCOne}, because we can reduce the quantified case to the propositional case. In \cite{Escape} this reduction was done in $\PA$ but it can actually be obtained in $\HA$.

The reduction boils down to the finite model property of $\QRC$ (Theorem~\ref{theorem:QRCcomplete}), which implies that a seemingly unrestricted universal quantification is equivalent to a finite conjunction. This is observed in Lemma~\ref{theorem:quantificationIsConjunction} below. As such, the complexity of the arithmetical interpretations of our quantified modal formulas can be restricted, as observed in Lemma~\ref{theorem:substitutionsForLpLow}.

To obtain these results, we first recall the arithmetical interpretations used in the arithmetical completeness proof presented in \cite{Escape}, where finite first-order Kripke models are embedded in arithmetic.

For the remainder of this section, we fix a certain finite constant domain model $\M$.
Let the domain of every world of $\M$ be denoted by $M$. Let $m$ be the size of $M$, and $\bij{\cdot}$ be a bijection between $M$ and the set of numerals $\{\overline{0}, \hdots, \overline{m - 1}\}$.

For a given $n$-ary predicate symbol $S$ and terms $\vec{t} := t_0, \hdots, t_{n-1}$, the $\cdot\ofs$ interpretation of $S$ is defined as follows:
  \begin{gather*}
    S(\vec{t})\ofs :=
    \bigvee_{i \leq N} \left(\lambda_i \land \Phi^{S(\vec{t})}_i \right) \text{, where:}\\
    \Phi_i^{S(\vec{t})} :=
          \bigvee_{\la a_0, \hdots, a_{n-1} \ra \in S^{J_i}}
          \bigwedge_{l < n} \left(
            \bij{a_l} =
              \begin{cases}
                y_k \umod \overline m &\text{if } t_l = x_k\\
                z_k \umod \overline m &\text{if } t_l = c_k
              \end{cases}
          \right)
      .
  \end{gather*}
Here $x_k$ is any $\QRC$ variable, $y_k$ is the corresponding arithmetical variable, $c_k$ is any $\QRC$ constant, $z_k$ is the corresponding arithmetical variable, and $\cdot^{J_i}$ is the denotation of a predicate symbol at world $i$.
It is clear that $\cdot\ofs$ depends on the particular choice of $\M$, but this dependency is omitted in the notation for simplicity. We keep $\M$ fixed throughout this section, so this should not lead to confusion.

\begin{remark}\label{remark:baseCaseSimple}
  For any given finite and constant domain model $\M$, we have that $S(\vec{t})\ofs \in \DCone$.
\end{remark}

The $\cdot\ofs$ interpretation is extended to all strictly positive formulas as described in Definition~\ref{def:QML:realization}. The interpretation $S(\vec{t})\ofs$ depends only on the values of its free variables modulo $m$. This modular dependence propagates over all formulas, as shown by the following lemma.

\begin{lemma}\label{theorem:quantificationIsConjunction}
\label{lem:mod}
  For any c.e.~theory $T$, any $\QRC$ formula $\varphi$, and any arithmetical variable $u$:
  \begin{enumerate}
    \item
      $
        \HA \vdash
          \varphi\ofsT
          \leftrightarrow
          \varphi\ofsT\subst{u}{(u \umod \overline m)}
      $;
    \item
      $
        \HA \vdash
          \Forall{u} \varphi\ofsT
          \leftrightarrow
          \bigwedge_{u \in \{\overline 0, \ldots , \overline{m-1}\}} \varphi\ofsT
      $.
  \end{enumerate}
\end{lemma}
\begin{proof}
  The second item is a straightforward consequence of the first, which is proved by an external induction on the complexity of $\varphi$, noting that $\HA$ proves:
  \begin{equation*}
    (u \umod \overline m) \umod \overline m = u \umod \overline m
  \end{equation*}
for any $u$.
\end{proof}

Since quantification does not essentially increase complexity, we see that the arithmetical interpretations of $\Lp$ are all of low complexity.

\begin{lemma}\label{theorem:substitutionsForLpLow}
For any c.e.~theory $T$ and any s.p.~formula $\varphi \in \Lp$, we have that $\varphi^{\ofsT} \in \DCone$.
\end{lemma}
\begin{proof}
  By an easy induction using Remark~\ref{remark:baseCaseSimple} for the base case and Lemma~\ref{theorem:quantificationIsConjunction} for the only non-trivial inductive step.
\end{proof}

We can now obtain the analogous of Lemma~\ref{theorem:HARealisationEquivalentToPA} in the quantified case. Note however that this does not hold in general for all realizations; the claim here is only for the specific realization $\cdot\ofs$.

\begin{lemma}\label{theorem:HAQrealisationEquivalentToPA}
  For any formula $\varphi \in \Lp$, we have $\HA \vdash \varphi\ofsHA \leftrightarrow \varphi\ofsPA$.
\end{lemma}
\begin{proof}
  By an easy external induction using Lemma~\ref{theorem:DiamondSigmaTwoConservativity} for the diamond case and Lemma~\ref{theorem:quantificationIsConjunction} for the quantifier case.
\end{proof}

Finally, we can be more specific in the statement of the completeness theorem $\QPLsp(\PA) \subseteq \QRC$.

\begin{theorem}[\protect{\cite[Theorem~8.7]{Escape}}]
\label{theorem:specific_finitary_completeness}

Let $\varphi$ and $\psi$ be $\QRC$ formulas and let $\vec y, \vec z$ represent all the free variables of $\varphi\ofsPA \land \psi\ofsPA$. If $\varphi \nvdash_{\QRC} \psi$, then there are standard numbers $\vec{y_0}, \vec{z_0}$ such that $\PA \nvdash (\varphi\ofsPA \to \psi\ofsPA)\subst{\vec y}{\vec{y_0}}\subst{\vec z}{\vec{z_0}}$.
\end{theorem}

With these preparations, we repeat our propositional completeness proof and obtain a quantified completeness proof.

\begin{theorem}\label{theorem:QPLpOfHAisQRC}
  Let $\varphi$ and $\psi$ be formulas in $\Lp$. We have $\varphi \vdash_{\QRC} \psi$ if and only if for all realizations $\cdot\fs$ we have
  $\HA \vdash \varphi\fsHA \to \psi\fsHA$.
\end{theorem}
\begin{proof}
  Soundness is proven by a straightforward induction on $\varphi \vdash_{\QRC} \psi$ and was already observed in \cite[Theorem~5.4.3]{BorgesPhD}. For completeness we reason as follows:
  \begin{align*}
    \varphi \nvdash_{\QRC} \psi
      &\implies
        \PA \nvdash (\varphi\ofsPA \to \psi\ofsPA)\subst{\vec y}{\vec{y_0}}\subst{\vec z}{\vec{z_0}}
        &&\text{by Theorem~\ref{theorem:specific_finitary_completeness}}
        \\
      &\implies
        \PA \nvdash (\varphi\ofsHA \to \psi\ofsHA)\subst{\vec y}{\vec{y_0}}\subst{\vec z}{\vec{z_0}}
        &&\text{by Lemma~\ref{theorem:HAQrealisationEquivalentToPA}}
        \\
      &\implies
        \HA \nvdash (\varphi\ofsHA \to \psi\ofsHA)\subst{\vec y}{\vec{y_0}}\subst{\vec z}{\vec{z_0}}
        &&\text{by \eqref{equation:HAincludedInPAFormalized},}
  \end{align*}
where $\vec y, \vec z$ are the free variables of $\varphi\ofsPA$ and $\psi\ofsPA$ (which are the same as the free variables of $\varphi\ofsHA$ and $\psi\ofsHA$) and $\vec{y_0}, \vec{z_0}$ are as given by Theorem~\ref{theorem:specific_finitary_completeness}. Thus there is a realization $\cdot\ofs$ such that $\HA \nvdash \Forall{\vec y, \vec z} (\varphi\ofsHA \to \psi\ofsHA)$, as we wanted to show.
\end{proof}

We now determine $\IQPL(\HA)$, the quantified reflection logic of $\HA$. Just as in the propositional case, we can interpret a finitary realization $\cdot\fs$ as an infinitary one by taking $S(\vec t)\is := (u = \gnum{S(\vec t)\fs})$. However, unlike in the propositional case, the realizations $\varphi\isT$ and $\tau \lor (u = \gnum{\varphi\fs})$ are not necessarily equivalent (contrast with Lemma~\ref{lemma:WDTST}).
In general, one can only prove $T \vdash \Forall \theta \Forall{\vec y, \vec z} (\Box_{\varphi\isT} \theta \to \Box_{\tau} (\varphi\fsT \to \theta))$ (see \cite[Lemma~5.2.8]{BorgesPhD} and the following discussion), but in the specific case of the realization $\ofs$ above we can indeed prove the equivalence.

Let $\cdot\ois$ be the infinitary counterpart to $\cdot\ofs$, i.e., $S(\vec t)\ois := (u = \gnum{S(\vec t)\ofs})$.
Note that $\varphi\oisT$ has same free variables as $\varphi\ofsT$ for any $\varphi$.

\begin{lemma}
\label{lem:same_thing}

  Let $\eta$ denote a standard axiomatization of $\HA$ and let $\varphi$ be a $\QRC$ formula with free variables $\vec{x}$ and constants $\vec{c}$. Then:
  \begin{equation*}
    \HA \vdash
      \Forall{\theta} \Forall{\vec y, \vec z} (
        \Box_{\varphi\oisHA} \theta \leftrightarrow \Box_\HA (\varphi\ofsHA \to \theta)
      )
    ,
  \end{equation*}
  where $\theta$ is (the Gödel number of) a closed formula.
\end{lemma}
\begin{proof}
  By external induction on $\varphi$. The proof is presented for $\PA$ in \cite[Lemma~8.8]{Escape} and adapts without any problem to $\HA$.
\end{proof}

With this reduction from quantified to propositional we can now determine the quantified reflection logic of $\HA$.

\begin{theorem}\label{theorem:QRefLogicHAisQRC}
$\IQPL(\HA) = \QRC$.
\end{theorem}
\begin{proof}
  The inclusion $\QRC \subseteq \IQPL(\HA)$ is a straightforward induction and is proved in \cite[Theorem~10.3]{Escape}. For the other direction, we reason as in the proof of Theorem~\ref{theorem:RefLogicHAisRCOne} and assume that $\varphi \nvdash_{\QRC} \psi$.

  By Theorem~\ref{theorem:QPLpOfHAisQRC} we find $\cdot\ofs$ and $\vec{y_0}, \vec{z_0}$ such that:
  \begin{equation}\label{equation:bla}
    \HA \nvdash (\varphi\ofsHA \to \psi\ofsHA)\subst{\vec y}{\vec{y_0}}\subst{\vec z}{\vec{z_0}}.
  \end{equation}
  By Lemma~\ref{lem:same_thing} we get for the corresponding $\cdot\oisHA$ that for arbitrary $\vec y, \vec z$ we have $\HA \vdash \Box_{\psi\oisHA} \psi\ofsHA$. Again by Lemma~\ref{lem:same_thing} and using \eqref{equation:bla}, in complete analogy with the proof of Theorem~\ref{theorem:RefLogicHAisRCOne}, we see that $\HA \nvdash \Box_{\varphi\oisHA} \psi\ofsHA$, so $\HA \nvdash \Forall{\theta} \Forall{\vec{y}, \vec{z}} (\Box_{\psi\oisHA} \theta \to \Box_{\varphi\oisHA} \theta)$ as was to be shown.
\end{proof}

\section{On interpolation and the fixpoint theorem}\label{section:InterpolationAndFP}

The interpolation and fixpoint theorems are important properties of theories and logics. In this section we consider them in the setting of $\QRC$.

\subsection{Interpolation}

The predicate provability logic of Peano Arithmetic does not have the interpolation property (due to Vardanyan, see \cite{Boolos1993}). Since $\QPL(\PA)$ fails interpolation, it is important to ask whether our fragment does satisfy interpolation. It turns out that indeed, both $\RC_1$ and $\QRC$ have interpolation, and the proof is rather trivial.

\begin{theorem}
  Let $T$ be either $\RC$ or $\QRC$. If $\varphi \vdash_T \psi$, then there is a formula $\chi$ in the common language of $\varphi$ and $\psi$ such that both $\varphi \vdash_T \chi$ and $\chi \vdash_T \psi$.
\end{theorem}
\begin{proof}
By an easy induction on a proof of $\varphi \vdash_T \psi$ we observe that all symbols occurring in $\psi$ already occur in $\varphi$. Consequently, we can take $\chi := \psi$, that is, $\psi$ is the required interpolant.
\end{proof}

\subsection{Fixpoint theorem}

The fixpoint theorem for $\GL$ states that for each formula $\varphi$ where all appearances of $p$ occur under the scope of a modality, we can find a fixpoint $\chi$ of $\varphi$ where $p$ does not occur such that $\GL \vdash \boxdot (p \leftrightarrow \varphi) \to (p \leftrightarrow \chi)$ (where $\boxdot \delta := \delta \land \Box \delta$).

We observe that the fixpoint theorem cannot really be formulated in our framework. The best we could do is to demand both $p \vdash \varphi(p, \vec q)$ and $\varphi(p, \vec q) \vdash p$ as antecedents (with $p$ occurring modalized in $\varphi$) of our version of the fixpoint theorem. Note, however, that $p \nvdash \varphi(p, \vec q)$ for any $\varphi$ where $p$ occurs modalized.
This is due to the following elementary observation.

\begin{lemma}[\protect{\cite[Lemma~2.4]{QRC1}}]
  Let $\mdepth(\varphi)$ be the modal depth (maximum number of nested diamonds) of $\varphi$. If $\varphi \vdash_{\RC_1} \psi$ or $\varphi \vdash_{\QRC} \psi$, then $\mdepth(\varphi) \geq \mdepth(\psi)$.
\end{lemma}
\begin{proof}
By an easy induction on the length of a $\RC_1$ or $\QRC$ proof.
\end{proof}

Thus, $p$ cannot prove any formula $\varphi$ that is modalized in $p$ (since the only available modality is the diamond). Our framework is hence not expressive enough to formulate a sensible version of the fixpoint theorem. Alternatively, we could state that the fixpoint theorem holds vacuously true.

\section*{Acknowledgments}

We are very grateful to Dick de Jongh, Mojtaba Mojtahedi, and Albert Visser, who helped with the intuitionistic mathematics involved in this work. Their reading, comments, suggestions and discussions substantially helped the presentation of this paper.


\providecommand\noopsort[1]{}\providecommand{\noopsort}[1]{}\providecommand{\noopsort}[1]{}

\end{document}